\documentclass[final]{amsart}

\usepackage[utf8]{inputenc} 

\usepackage[foot]{amsaddr} 
\usepackage{verbatim} 
\usepackage{amssymb}
\usepackage{amsthm}
\usepackage{amsmath}
\usepackage{amsfonts}
\usepackage{latexsym}
\usepackage{mathtools} 
\usepackage{enumitem} 
\usepackage{cite} 
\usepackage[english]{babel} 

\usepackage{esint}

\usepackage{calrsfs}

    \usepackage[backref=none]{hyperref} 
     \usepackage[usenames,dvipsnames]{color}
     	\definecolor{columbiablue}{rgb}{0.61, 0.87, 1.0}	\definecolor{richblack}{rgb}{0.0, 0.25, 0.25}
\definecolor{pennblue}{RGB}{1,31,91} 
\definecolor{pennred}{RGB}{153,0,0}	  
     \definecolor{MyDarkBlue}{rgb}{0,0.1,0.7}

\usepackage[color]{showkeys}

\usepackage{color}

\theoremstyle{plain}
\newtheorem{theorem}{Theorem}
\newtheorem{lemma}[theorem]{Lemma}

\newtheorem{proposition}[theorem]{Proposition}
\newtheorem{remark}[theorem]{Remark}

\newtheorem{corollary}[theorem]{Corollary}

\numberwithin{equation}{section}
\numberwithin{theorem}{section}

\allowdisplaybreaks

\newcommand{\eqdef }{\overset{\mbox{\tiny{def}}}{=}}
\newcommand{\rth}{{\mathbb{R}^3}}

\title[Propagation of $L^p$ estimates]{Propagation of $L^p$ estimates for the Spatially Homogeneous Relativistic Boltzmann Equation}

\author[J. W. Jang]{Jin Woo Jang$^\dagger$}
\address{$^\dagger$Center for Geometry and Physics, Institute for Basic Science (IBS), Pohang 37673, Republic of Korea. \href{mailto:jangjinw@ibs.re.kr}{jangjinw@ibs.re.kr} }

\author[S.-B. Yun]{Seok-Bae Yun$^\ddagger$}
\address{$^\ddagger$Department of mathematics, Sungkyunkwan University, Suwon 16419, Republic of Korea. \href{mailto:sbyun01@skku.edu}{sbyun01@skku.edu} }

\AtBeginDocument{%
   \def\MR#1{}
}

\begin{document}


\date{\today}

\let\thefootnote\relax
\addtocounter{footnote}{-1}\let\thefootnote\svthefootnote

\begin{abstract}
In this paper, we prove the propagation of $L^p$ upper bounds for the spatially homogeneous relativistic Boltzmann equation for any $1<p<\infty$. We consider the case of relativistic \textit{hard ball} with Grad's angular cutoff. Our proof is based on a detailed study of the interrelationship between the relative momenta, the regularity and the $L^p$ estimates for the gain operator, the development of the relativistic Carleman representation, and several estimates on the relativistic hypersurface $E^{v_*}_{v'-v}$. We also derive a Pythagorean theorem for the \textit{relative momenta} $g(v,v_*),$ $g(v,v')$, and $g(v',v_*)$, which has a crucial role in the reduction of the momentum singularity.
\end{abstract}

\setcounter{tocdepth}{1}

\maketitle
\tableofcontents

\thispagestyle{empty}

\section{Introduction}The relativistic Boltzmann equation is a central model in the kinetic theory of gases for the dynamics of fast moving particles whose speed is comparable to the speed of light. The equation is a generalized model of the classical Boltzmann equation and is based on Einstein's theory of special relativity.  
The spatially homogeneous relativistic Boltzmann equation reads
\begin{equation}
\label{homorelbol}
\frac{\partial f}{\partial t}(t,v) = Q(f,f)(t,v),\quad v\in \rth, \quad t\geq 0,
\end{equation}
where the unknown $f=f(t,v)$ is a probability density function on $\rth$ at time $t$. 

 The collision operator $Q(f,f)$ is a quadratic bilinear operator defined as
\begin{equation}
\label{collisionQ}
Q(f,h)(t,v)\eqdef  \int_\rth dv_*\int_{\mathbb{S}^2}d\omega\ v_{\o} \sigma(g,\theta) [f(t,v')h(t,v_*')-f(t,v)h(t,v_*)],
\end{equation}in the \textit{center-of-momentum} framework, where $\sigma(g,\theta)$ is a scattering kernel (or is sometimes called collision cross-section) that depends only on the \textit{relative momentum} $g$ and the scattering angle $\omega$, and $v_{\o}$ is the M\"oller velocity defined by $$v_{\o}=v_{\o}(v,v_*)=\sqrt{\left|\frac{v}{v^0}-\frac{v_*}{v_*^0}\right|^2-\left|\frac{v}{v^0}\times \frac{v_*}{v_*^0}\right|^2}=\frac{g\sqrt{s}}{v^0v^0_*}.$$ Here, $v^0$ stands for the energy of a relativistic particle with momentum $v$ and is defined as $v^0\eqdef \sqrt{1+|v|^2}$, where, for the sake of simplicity, we normalize the speed of light and the rest mass to be 1. We define $v_*^0$, $v'^0$, and $v'^0_*$ in a similar manner. The definition implies that $v_{\o}\lesssim 1.$ 

The notations $g$ and $s$ stand for the \textit{relative momentum} and the square of energy in the \textit{center-of-momentum} frame and are defined by
\begin{equation}\label{g}
g=g(v,v_*)\eqdef \sqrt{-(v^0-v^0_*)^2+|v-v_*|^2}\geq 0,
\end{equation} and 
\begin{equation}\label{s}
s=s(v,v_*)\eqdef (v^0+v^0_*)^2-|v+v_*|^2.
\end{equation} Then we have the identity $s=g^2+4.$ We also define the notations $\bar{g}\eqdef g(v,v')$, $\tilde{g}\eqdef g(v',v_*)$, $\bar{s}\eqdef s(v,v')$, and $\tilde{s}\eqdef s(v',v_*)$ in a similar manner.

In the definition of the collision operator $Q(f,f)$, $v'$ and $v'_*$ stand for the post-collisional momenta of particles which had the momenta $v$ and $v_*$ before the collision, respectively. The equation was derived under the assumption that the momenum and the energy are conserved after each collision; i.e., \begin{equation}\label{emcons}v+v_*=v'+v'_*\quad \text{and}\quad v^0+v^0_*=v'^0+v'^0_*.\end{equation} In the \textit{center-of-momentum} frame, the post-collisional momenta can be represented as
\begin{equation}
\label{postcol}
\begin{split}
v'&\eqdef\frac{v+v_*}{2}+\frac{g}{2}\left(\omega+(\gamma-1)(v+v_*)\frac{(v+v_*)\cdot \omega}{|v+v_*|^2}\right),\\
v'_*&\eqdef\frac{v+v_*}{2}-\frac{g}{2}\left(\omega+(\gamma-1)(v+v_*)\frac{(v+v_*)\cdot \omega}{|v+v_*|^2}\right),\\
\end{split}
\end{equation}
where $\gamma\eqdef \frac{v^0+v_*^0}{\sqrt{s}}.$ In the same frame, the microscopic energy of the post-collisional momenta $v'^0$ and $v'^0_*$ are represented by
\begin{equation}
\label{postcolen}
\begin{split}
v'^0&\eqdef\frac{v^0+v_*^0}{2}+\frac{g}{2\sqrt{s}}(v+v_*)\cdot \omega,\\
v'^0_*&\eqdef\frac{v^0+v_*^0}{2}-\frac{g}{2\sqrt{s}}(v+v_*)\cdot \omega.
\end{split}
\end{equation}The scattering angle $\theta$ is defined by 
\begin{equation}\label{costheta}\cos\theta = \frac{-(v^0-v_*^0)(v'^0-v'^0_*)+(v-v_*)(v'-v'_*)}{g^2}.\end{equation} The angle $\theta$ is well-defined \cite{Gla} together with the energy and momentum conservation laws \eqref{emcons}. 

The formal conservation laws for the mass, the momentum, and the energy of the system read 
\begin{equation}\label{conlaw}
\int_\rth  \begin{pmatrix}1\\v\\v^0 \end{pmatrix} f(t,v)dv=\int_\rth\begin{pmatrix}1\\v\\v^0 \end{pmatrix}f_0(v) dv.
\end{equation} Also, we mention the formal Boltzmann $H$-theorem
\begin{equation}\label{entropy.eq}
H(f(t))+\int^t_0D(f(s))ds\leq H(f_0),
\end{equation}
where the entropy functional is defined by
\begin{equation}\label{entropy.functional}
H(f(t)) = \int_\rth f(t,v)\ln f(t,v)dv.
\end{equation}
\subsection{A short list of the previous results on the relativistic collisional kinetic theory}
This section is devoted to a short list of the previous results in the relativistic collisional kinetic theory. We start with providing a list of general references on the relativistic kinetic theory: \cite{C,C-I-P,C-K,E-M-V,DeGroot,Gla,Vil02}.

For the spatially homogeneous relativistic Boltzmann equation, we have the existence theory developed in \cite{MR3169776,MR3166961}. The gain of regularity for the gain operator was shown in \cite{MR3880739}. For the spatially homogeneous Landau equation, we have the entropy dissipation estimate \cite{StrainTas} and the conditional uniqueness \cite{1903.05301}.

There have been many other varied developments in the spatially inhomogeneous relativistic Boltzmann equation. These include the exitence theory \cite{D-E3,D,GS3,GS4,GL-Vacuum,D-E0,D-E0er,D-E2,MR2378164}, the derivation of the center-of-momentum representation of the operator \cite{DeGroot,MR2765751}, the blow-up theory in the case without the loss term \cite{MR2102321}, the Newtonian limit of the equation \cite{Cal,MR2679588}, and the regularizing effect of the operator \cite{MR1402446, MR3880739, MR1450762}. For the existence theory on the relativistic Landau and the relativistic Boltzmann equation coupled with the Maxwell equations, we remark the work of \cite{Guo-Strain3, Guo-Strain2}. There are also some results on the relativistic BGK-type models \cite{MR2988960,MR3300786,1801.08382,1811.10023,Pennisi_2018,H-R-Yun}.

\subsection{Hypothesis on the collision cross-section}
The relativistic collision cross-section $\sigma(g,\theta)$ is a non-negative function which only depends on the relative velocity $g$ and the scattering angle 
$\theta$. We assume that $\sigma$ takes the form of the product in its arguments; i.e., $$\sigma(g,\theta)\eqdef \Phi(g)\sigma_0(\theta).$$ In general, we suppose that both $\Phi$ and $\sigma_0$ are non-negative functions.

Without loss of generality, we may assume that the collision kernel $\sigma$ is supported only when $\cos\theta\geq 0$ throught this paper;
i.e., $0\leq \theta \leq \frac{\pi}{2}$. 
Otherwise, the following \textit{symmetrization} \cite{Gla} will reduce the case:
$$
\bar{\sigma}(g,\theta)=[\sigma(g,\theta)+\sigma(g,-\theta)]1_{\cos\theta\geq 0},
$$
where $1_A$ is the indicator function of the set $A$. 
We further assume the collision kernel satisfies the following hard potential assumption:
\begin{equation}
\label{hard}
\Phi(g) =  C_{\Phi} g, \quad C_{\Phi}>0.
\end{equation}
We suppose that the angular function $\theta \mapsto \sigma_0(\theta)$ is non-negative and satisfies for some $C>0$ that
\begin{equation}
\label{angassumption}\sigma_0(\theta)=C\sin \theta, \hspace*{5mm}\forall \theta \in \Big(0,\frac{\pi}{2}\Big].
\end{equation}This is also called the relativistic \textit{hard-ball} assumption with a Grad's angular cutoff assumption. See \cite[Appendix B]{MR2679588} for more detailed physical descriptions on the collision cross-sections in relativistic kinetic theory. 

Together with the assumption on the collision cross-section, we can further split the collision operator $Q(f,f)$ into two and further define them as the gain term
\begin{equation}
\label{gain}
Q^+(f,h)\eqdef  \int_\rth dv_*\int_{\mathbb{S}^2}d\omega\ v_{\o} \sigma(g,\theta) f(v')h(v_*'),
\end{equation} and the loss term\begin{equation}
\label{loss}
Q^-(f,h)\eqdef  \int_\rth dv_*\int_{\mathbb{S}^2}d\omega\ v_{\o} \sigma(g,\theta) f(v)h(v_*),
\end{equation} where we have used the shorthand $f(v)=f(t,v)$. So, we have $Q=Q^+-Q^-$. Then the loss term can further be written as
$$Q^-(f,h)=f(v)\int_\rth dv_*\int_{\mathbb{S}^2}d\omega\ v_{\o} \sigma(g,\theta) h(v_*)=: f(v)Lh(v),$$ with a new operator $L$. \subsection{Spaces}In this section, we define the weighted $L^p$ spaces for $p\in [1,\infty)$. We define the weighted $L^p$ space $L^p_k$ as the space of functions whose $L^p_k$ norm defined as below is bounded:
$$||f||_{L^p_k}\eqdef \left( \int_\rth dv \ (v^0)^{k} |f(v)|^p\right)^{\frac{1}{p}}.$$ For $L^\infty_k$, we use the following definition:
$$||f||_{L^\infty_k}\eqdef \sup _{v\in \rth}|f(v)|(v^0)^k.$$ 
\subsection{Main results}We may now state our main theorem.
\begin{theorem}\label{main}
	Let $p\in (1,+\infty)$ and $\eta>2$. Let $f_0$ be a nonnegative function satisfying
	\begin{equation}\label{initial}
	f_0\in L^1_m\cap L^p(\rth)\quad \text{and}\quad \int_\rth f_0|\log f_0|dv <\infty,
	\end{equation} with $m=m(\eta)>1$ defined in \eqref{mdef}. Suppose the collision cross-section satisfies \ref{hard} and \eqref{angassumption}. Let the solution $f(t,v)$ to the spatially homogeneous relativistic Boltzmann equation \eqref{homorelbol} with initial datum $f_0$ have finite mass, energy, and entropy as in \eqref{conlaw} and \eqref{entropy.eq}. Then $f$ satisfies the differential inequality 
	$$\frac{d}{dt}||f||^p_{L^p}\leq C_1||f||^{q}_{L^p}-C_2||f||^p_{L^p_{1}},$$ for some $0<q<p.$
	As a consequence, we have that for all $t\geq 0$, 
	$$||f(t,\cdot)||_{L^p}\leq C_{f_0},$$ where $C_{f_0}$ depends only on $p$ and the initial conditions $||f_0||_{L^1_m}$ and $||f_0||_{L^p}$.
\end{theorem}
We now briefly discuss the main theorem. Historically, the problem of showing the propagation of the uniform $L^p$ upper bound of the solution has been widely open for the relativistic Boltzmann equation. This is due to the extremely complicated structure of the representations of the post-collisional momenta \eqref{postcol} in constrast to the relatively simpler post-collisional velocities in the Newtonian case. Because of the complexity, it has been very limited to use the change of pre-post collisional variables $v\mapsto v'$ or vice versa, as the Jacobian is no longer uniformly bounded above and below in the relativistic scenario. This was studied in \cite{Jang2016} and \cite{Jang-Strain}. However, in order to deal with the dual formation like \eqref{duality}, one must find an equivalent way to the change of variables $v\mapsto v'$ in the estimate.

In order to resolve this issue, we adapt the relativistic analogue of the Carleman representation that was recently developed and used in \cite{Jang2016} and \cite{Jang-Strain-Yun}. By considering a relativistic counterpart of the Carleman representation, we have a chance to convert the integration with respect to $dv\times dv_*\times d\omega$ on $\rth\times\rth\times \mathbb{S}^2$ into the integration with respect to $dv\times dv'\times d\pi_{v_*}$ on $\rth\times \rth\times E^{v_*}_{v'-v}$ where $E^{v_*}_{v'-v}$ is a noncompact 2-dimensional hypersurface and $d\pi_{v_*}$ is the Lebesgue measure on the surface. The use of the relativistic Carleman representation inevitably involves the integration on the hypersurface $E^{v_*}_{v'-v}$ and the estimate on the nontrivial hypersurface was motivated by the work in \cite{Jang2016}. 

Unfortunately, it turned out that the use of the Carleman-type representation involves the appearance of an additional momentum singularity of $\frac{1}{g(v,v')}$ as we observe in Lemma \ref{carlemanrep}. In order to treat the additional singularity, we derive the Pythagorean theorem for the relative momenta $g(v,v_*),$ $g(v,v')$, and $g(v',v_*)$ as $$g(v,v_*)^2=g(v,v')^2+g(v',v_*)^2.$$ Then the Pythagorean identity can further be used to relate the angular part of the scattering kernel $\sigma_0(\theta)$ to $\frac{g(v,v')}{g(v,v_*)}$ via the \textit{symmetrization} \cite{Gla} and the reduction of the angular support of the collision kernel. Namely, we make use of the angular part of the scattering kernel to reduce the momentum singularity. We would like to mention that this is similar to the techniques that one creates and uses additional \textit{relative momenta} $|v-v'|$ to reduce the angular singularity of the non-cutoff \textit{inverse-power-law} Boltzmann equation the other way round, in the sense that we reduce the momentum singularity via the use of the angular part and the Pythagorean theorem. Indeed, the use of the angular part to reduce the momentum singularity was crucial for our \textit{angular cutoff} situation and, to the best of authors' knowledge, there has been no such a method used for the mathematical analysis of the Boltzmann equation.

We remark that this work generalizes the $L^p$ propagation theory in \cite{MR2081030} for the spatially homogeneous classical Boltzmann equation with an angular cutoff.  We also mention here that the weight $m$ in Theorem \ref{main} is not optimal. 
\subsection{Outline of the paper} The paper is organized as follows. In Section \ref{prelim}, we introduce preliminary lemmas for the proof of the propagation of $L^p$ bounds. This will include an introduction to the derivation of the Pythagorean theorem for the relative momenta in the relativistic collisional kinetic theory in Proposition \ref{lemma:pytha}. Based on the preliminary lemmas, we will establish an upper-bound estimate for the $L^p$ norm of the gain operator $Q^+(f,h)$ in Section \ref{Lp section}. Then the use of the Riesz-Thorin interpolation theorem will improve the $L^p$ estimate on the gain operator $Q^+$. Consequently, in Section \ref{Section proof of the main theorem}, we prove the main theorem, Theorem \ref{main}, on the propagation of $L^p$ estimates using the improved $L^p$ upper-bound estimate for the gain operator $Q^+$ from Section \ref{Lp section} and the lower-bound estimate for the loss operator $Q^-$ from Lemma \ref{lowerL}.
\section{Preliminary lemmas}\label{prelim} We start with introducing many preliminary estimates on the relativistic terms, which will be crucially used in the paper. We start with a well-known coercive inequality for the \textit{relative momentum} $g$ in the \textit{center-of-momentum} framework.

\begin{lemma}[Lemma 3.1 (i) and (ii) on page 316 of \cite{GS3}]\label{coersive inequality}  The relative momentum $g$ satisfies the following inequalities:
\begin{equation}\label{gINEQ}
\frac{|v-v_*|}{\sqrt{v^0v_*^0 }}\leq g(v,v_*)\leq |v-v_*|.
\end{equation} Moreover, $g^2<g\sqrt{s}\lesssim v^0v_*^0.$
\end{lemma}
\begin{remark}
We remark that in \cite{GS3} $g$ is defined as $\frac{1}{2}$ of our $g$ in \eqref{g}. \end{remark}
 We will also make a use of the following uniform lower bound estimate for the loss term operator $L$:
\begin{lemma}[Lemma 3.3 on page 925 of \cite{MR3166961}]\label{lowerL}
	Let $f(t,v)$ have finite mass, energy, and entropy as in \eqref{conlaw}, \eqref{entropy.eq}, and \eqref{initial}. Then there exists uniform positive constants $C_{\ell}>0$ and $C_u>0$, which are determined only by the mass, energy, and entropy of the initial data $f_0$, such that the following estimate holds:
$$
C_{\ell} v^0 \leq 
\int_{\rth\times \mathbb{S}^2}v_{\o} \sigma(g,\theta)f(t,v_*)d\omega dv_*
= (Lf)(t,v)
\leq C_u v^0.
$$
\end{lemma}One of the most crucial identities that we use in this paper is the following identity on the representations of the relativistic Boltzmann collision operator. We call it a relativistic Carleman representation.
\begin{lemma}[The relativistic Carleman representation \cite{Jang2016}]\label{carlemanrep}For any suitable integrable function $A:\rth\times\rth\times \rth \mapsto \mathbb{R}$, it holds that
\begin{multline*}
\int_\rth \frac{dv}{v^0}\int_\rth \frac{dv_*}{v_*^0}\int_\rth \frac{dv'}{v'^0}\int_\rth \frac{dv'_*}{v'^0_*}\ s\sigma(g,\theta)\\\times \delta^{(4)}\left((v^0,v)+(v_*^0,v_*)-(v'^0,v')-(v'^0_*,v'_*)\right)A(v,v_*,v')
\\=\int_\rth \frac{dv}{v^0}\int_\rth \frac{dv'}{v'^0}\int_{\rth} \frac{d\pi_{v_*}}{v_*^0}\ \frac{s}{2\bar{g}}\sigma(g,\theta)A(v,v_*,v'),
\end{multline*} where $$d\pi_{v_*}\eqdef dv_*\ u(v^0+v_*^0-v'^0)\delta\left(\frac{\bar{g}}{2}+\frac{-v_*^0(v^0-v'^0)+v_*\cdot(v-v')}{\bar{g}}\right),$$ and $u(x)=1$ if $x\geq 1$ and $=0$ otherwise.
\end{lemma}\begin{proof}
The proof is given in Section 2.2 of \cite[page 17, see (2.2.3)]{Jang2016}. 
\end{proof}
For the interpolation between Lebesgue spaces, we will record here a lemma on the gain of regularity for the relativistic Boltzmann gain operator.
\begin{lemma}[Gain of regularity \cite{MR3880739}, Theorem 1.1]
$$||\nabla_v Q^+(f,h)||_{L^2}\lesssim ||f||_{L^1}||h||_{L^2}.$$\end{lemma} This immediately implies the following corollary:\begin{corollary}\label{Gain of integrability}
$$||Q^+(f,h)||_{L^6}\lesssim ||f||_{L^1}||h||_{L^2}.$$\end{corollary} \begin{proof}
This is almost obvious via the Sobolev embedding theorem $H^1\subset L^6.$
\end{proof}Finally, we derive the Pythagorean theorem for the relative momenta as in the following proposition:
\begin{proposition}[Pythagorean theorem for the \textit{relative momenta}]\label{lemma:pytha}The relative momenta $g=g(v,v_*),$ $\bar{g}=g(v,v')$, and $\tilde{g}=g(v',v_*)$ satisfy the following identity$$g^2=\tilde{g}^2+\bar{g}^2.$$ Moreover, $$\sin\frac{\theta}{2}=\frac{\bar{g}}{g},$$ if $\theta$ is defined as in \eqref{costheta}.
\end{proposition}\begin{proof}
We first claim that the momentum and energy conservation laws \eqref{emcons} imply that $$\bar{g}^2-2v_*^0(v^0-v'^0)+2v_*\cdot(v-v')=0.$$ 
In order to prove the claim, we first observe that 
$$-(v_*^0)^2+|v_*|^2=-1=-(v'^0_*)^2+|v'_*|^2.$$ Therefore, this implies
$$-(v_*^0+v'^0_*)(v_*^0-v'^0_*)+(v_*+v'_*)\cdot(v_*-v'_*)=0,$$ and this is equivalent to 
$$-(v_*^0+v'^0_*)(v^0-v'^0)+(v_*+v'_*)\cdot(v-v')=0,$$ by the conservation laws  \eqref{emcons}. Then, we use $$(v'^0_*,v'_*)= (v^0+v^0_*-v'^0,v+v_*-v')$$ of \eqref{emcons} and obtain that$$-(v^0+2v_*^0-v'^0)(v^0-v'^0)+(v+2v_*-v')\cdot(v-v')=0.$$ Then we recall the definition of $\bar{g}$ that 
$$\bar{g}=g(v,v')\eqdef \sqrt{-(v^0-v'^0)^2+|v-v'|^2},$$ and use this definition to obtain that 
$$\bar{g}^2-2v_*^0(v^0-v'^0)+2v_*\cdot(v-v')=0.$$
This proves the claim. Now we observe that
\begin{multline*}
\bar{g}^2=-(v^0-v'^0)^2+|v-v'|^2\\
= (-(v^0)^2+|v|^2)+(-(v'^0)^2+|v'|^2)-2(-v^0v'^0+v\cdot v'),
\end{multline*} by expanding the squares. Then, we have \begin{multline*}
0=	\bar{g}^2-2v_*^0(v^0-v'^0)+2v_*\cdot(v-v')\\
	=(-(v^0)^2+|v|^2)+(-(v'^0)^2+|v'|^2)-2(-v^0v'^0+v\cdot v')\\+2(-v_*^0v^0+v_*\cdot v)-2(-v_*^0v'^0+v_*\cdot v').
	\end{multline*}
	Now we use $$(-(v^0)^2+|v|^2)=-1=(-(v'^0)^2+|v'|^2)$$ to further continue as
	\begin{multline*}
	0=2(-(v'^0)^2+|v'|^2)-2(-v^0v'^0+v\cdot v')+2(-v_*^0v^0+v_*\cdot v)-2(-v_*^0v'^0+v_*\cdot v')\\
	=2\left(-(v'^0-v_*^0)(v'^0-v^0)+(v'-v_*)\cdot (v'-v)\right).
\end{multline*} Therefore, we have
\begin{equation*}\begin{split}
	g^2&=-(v^0-v_*^0)(v^0-v_*^0)+(v-v_*)\cdot (v-v_*)\\
	&=-(v^0-v'^0+v'^0-v_*^0)(v^0-v'^0+v'^0-v_*^0)\\
	&\quad+(v-v'+v'-v_*)\cdot (v-v'+v'-v_*)\\
	&=-(v^0-v'^0)(v^0-v'^0)+(v-v')\cdot (v-v')\\&\quad+2\left(-(v'^0-v_*^0)(v^0-v'^0)+(v'-v_*)\cdot (v-v')\right)\\&\quad-(v'^0-v_*^0)(v'^0-v_*^0)+(v'-v_*)\cdot (v'-v_*)\\
	&=-(v^0-v'^0)(v^0-v'^0)+(v-v')\cdot (v-v')\\&\quad-(v'^0-v_*^0)(v'^0-v_*^0)+(v'-v_*)\cdot (v'-v_*)
	=\bar{g}^2+\tilde{g}^2.\end{split}
\end{equation*}This completes the proof of the Pythagorean theorem. Then, the relation $\sin\frac{\theta}{2}=\frac{\bar{g}}{g}$ immediately follows by the Pythagorean theorem above, the identity $$\cos\theta = \frac{-\bar{g}^2+\tilde{g}^2}{g^2}$$ from Page 12 of \cite{Jang2016}, and the half-angle formula. This completes the proof.
\end{proof}Together with the assumption on the collision cross-section \eqref{hard} and \eqref{angassumption}, the lemma above implies the following corollary:
\begin{corollary}\label{barg}
	Suppose that the collision cross-section $\sigma(g,\theta)$ satisfies \eqref{hard} and \eqref{angassumption}.  Then, we have
	$$\sigma(g,\theta)\approx g\sin\left(\frac{\theta}{2}\right)\approx \bar{g}.$$
\end{corollary}
\begin{proof}
	Since we assume that $\theta\in \left(0,\frac{\pi}{2}\right],$ we have $\sin\theta\approx \sin \frac{\theta}{2}.$ By Proposition \ref{lemma:pytha}, we observe $\sin \frac{\theta}{2}\approx \frac{\bar{g}}{g}$ and obtain $\sigma(g,\theta)\approx \bar{g}$. 
\end{proof}
\section{$L^p$ estimates on the gain operator}\label{Lp section}
This section is devoted to show an $L^p$ estimate for the gain term in the collision operator. More precisely, we prove the following theorem:
\begin{theorem}\label{Lp estimate Q+}
Let $p\in [1,\infty)$ and $\eta>2.$ Let the collision cross-section $\sigma(g,\theta)$ satisfies the assumption \eqref{hard} and \eqref{angassumption}. Then, we have
\begin{equation}
||Q^+(f,h)||_{L^p(\mathbb{R}^3)}\lesssim ||f||_{L^1_{1/2}(\rth)}||h||_{L^p_{\eta(p-1)+\frac{1}{2}}(\rth)}.
\end{equation}
\end{theorem}
\begin{proof}
By duality, we have
\begin{equation}\label{duality}||Q^+(f,h)||_{L^p(\mathbb{R}^3)}=\sup\left(\int_\rth Q^+(f,h)\psi(v) dv: ||\psi||_{L^{p'}(\rth)}\leq 1\right),\end{equation} where $\frac{1}{p}+\frac{1}{p'}=1.$ By definition, we have
$$\int_\rth Q^+(f,h)\psi(v) dv=\int_\rth dv\int_\rth dv_*\int_{\mathbb{S}^2}d\omega\ v_{\o}\sigma(g,\theta)f(v')h(v'_*)\psi(v).$$ Then we consider the pre-post change of variables $(v,v_*)\mapsto (v',v'_*)$ and obtain that
$$\int_\rth Q^+(f,h)\psi(v) dv=\int_\rth dv\int_\rth dv_*\int_{\mathbb{S}^2}d\omega\ v_{\o}\sigma(g,\theta)f(v)h(v_*)\psi(v'),$$ where we used $g(v,v_*)=g(v',v'_*)$ and $v_{\o}$ is invariant under the change of variables. Then by the H\"older inequality, we obtain
\begin{multline*}\int_\rth Q^+(f,h)\psi(v) dv\lesssim \int_\rth dv\int_\rth dv_*\int_{\mathbb{S}^2}d\omega\ v_{\o}\sigma(g,\theta)f(v)h(v_*)\psi(v')\\
\lesssim \left(\int_\rth dv\int_\rth dv_*\int_{\mathbb{S}^2}d\omega\ v_{\o}\sigma(g,\theta)|f(v)||h(v_*)|^p(v_*^0)^{\frac{\eta p}{p'}}\right)^{\frac{1}{p}}\\
\times \left(\int_\rth dv\int_\rth dv_*\int_{\mathbb{S}^2}d\omega\ v_{\o}\sigma(g,\theta)|f(v)||\psi(v')|^{p'}(v_*^0)^{-\eta}\right)^{\frac{1}{p'}},
\end{multline*} for some $\eta>2$ and $\frac{1}{p}+\frac{1}{p'}=1$. 
We first observe that $$\sigma(g,\theta)\lesssim g\lesssim \sqrt{v^0v'^0}.$$
Thus we further observe that
\begin{multline}\label{middlestep}\int_\rth Q^+(f,h)\psi(v) dv\\
\lesssim||f||_{L^1_{1/2}}^{\frac{1}{p}}||h||_{L^p_{\eta\frac{p}{p'}+\frac{1}{2}}} \left(\int_\rth dv\int_\rth dv_*\int_{\mathbb{S}^2}d\omega\ v_{\o}\sigma(g,\theta)|f(v)||\psi(v')|^{p'}(v_*^0)^{-\eta}\right)^{\frac{1}{p'}},
\end{multline} as we have $v_{\o}\lesssim 1$. We now define 
\begin{equation}\label{8foldint}K\eqdef \left(\int_\rth dv\int_\rth dv_*\int_{\mathbb{S}^2}d\omega\ v_{\o}\sigma(g,\theta)|f(v)||\psi(v')|^{p'}(v_*^0)^{-\eta}\right)^{\frac{1}{p'}}.\end{equation}

The estimate on $K$ will be performed via considering a new representation of $K$. This new representation is called as the relativistic Carleman representation, which was originally derived and used in \cite{Jang2016} and \cite{Jang-Strain-Yun}. Here we would like to make a brief introduction on the derivation of the relativistic Carleman representation of $K$. We first recall \cite[Theorem 2]{MR2765751} which states that
\begin{multline*}
 \int_{\mathbb{S}^2}d\omega\ v_{\o}\sigma(g,\theta) G(v,v_*,v',v'_*)=\frac{c}{v^0v_*^0}\int_\rth \frac{dv'}{v'^0}\int_\rth \frac{dv'_*}{v'^0_*}\ \frac{s}{2}\sigma(g,\theta)G(v,v_*,v',v'_*)\\\times \delta^{(4)}\left((v^0,v)+(v_*^0,v_*)-(v'^0,v')-(v'^0_*,v'_*)\right),
\end{multline*} for any suitable integrable function $G:\rth\times\rth\times\rth\times\rth \mapsto \mathbb{R}$,
where $c$ is the speed of light which has been normalized to be $1$ in the paper. Then we use this theorem and raise the 8-fold integral in \eqref{8foldint} to a 12-fold integral (see \cite{MR2765751, Jang2016, Jang-Strain-Yun, MR2728733}) by adding the 4-dimensional delta function of momentum and energy conservation laws: 
\begin{multline}\label{Kstep1}K\approx \bigg(\int_\rth \frac{dv}{v^0}\int_\rth \frac{dv_*}{v_*^0}\int_\rth \frac{dv'}{v'^0}\int_\rth \frac{dv'_*}{v'^0_*}\ s\sigma(g,\theta)|f(v)||\psi(v')|^{p'}(v_*^0)^{-\eta}\\\times \delta^{(4)}\left((v^0,v)+(v_*^0,v_*)-(v'^0,v')-(v'^0_*,v'_*)\right)\bigg)^{\frac{1}{p'}}.\end{multline} 
Then the use of Lemma \ref{carlemanrep} on the representation \eqref{Kstep1} admits a new representation of $K$ of
$$K\approx \bigg(\int_\rth \frac{dv}{v^0}\int_\rth \frac{dv'}{v'^0}\int_{\rth} \frac{d\pi_{v_*}}{v_*^0}\ \frac{s}{2\bar{g}}\sigma(g,\theta)|f(v)||\psi(v')|^{p'}(v_*^0)^{-\eta}\bigg)^{\frac{1}{p'}},$$where the Lebesgue measure on the hypersurface $E^{v_*}_{v'-v}$ is defined as $$d\pi_{v_*}\eqdef dv_*\ u(v^0+v_*^0-v'^0)\delta\left(\frac{\bar{g}}{2}+\frac{-v_*^0(v^0-v'^0)+v_*\cdot(v-v')}{\bar{g}}\right),$$ and $u(x)=1$ if $x\geq 1$ and $=0$ otherwise. This representation is what we call the relativistic Carleman representation of $K$.
 Since Corollary \ref{barg} additionally states that $\sigma(g,\theta)\approx \bar{g}$, we also obtain that the integral $K$ is equivalent to
$$K\approx \bigg(\int_\rth \frac{dv}{v^0}\int_\rth \frac{dv'}{v'^0}\int_{\rth} \frac{d\pi_{v_*}}{v_*^0}\ s|f(v)||\psi(v')|^{p'}(v_*^0)^{-\eta}\bigg)^{\frac{1}{p'}}.$$ 

In order to achieve an upper-bound estimate for the new represenation $K$, we first take a change of variables on $v_*$ into angular coordinates as $v_*\in \rth \mapsto (r,\theta,\phi)$ and choose the z-axis parallel 
	to $v-v'$ such that the angle between $v_*$ and $v-v'$ is equal to $\phi.$   The terms in the delta function can be rewritten as
	\begin{multline*}
		\frac{\bar{g}}{2}+\frac{-v_*^0(v^0-v'^0)+v_*\cdot(v-v')}{\bar{g}}=\frac{1}{2\bar{g}}(\bar{g}^2+2(-v_*^0(v^0-v'^0)+v_*\cdot(v-v')))\\
		=\frac{1}{2\bar{g}}(\bar{g}^2-2\sqrt{1+r^2}(v^0-v'^0)+2r|v-v'|\cos\phi).	
	\end{multline*}
		Thus, we obtain that 
			\begin{multline*}	
			\int_{\rth}\frac{d\pi_{v_*}}{v_*^0} s(v_*^0)^{-\eta}\lesssim v^0\int_{\rth}d\pi_{v_*}(v_*^0)^{-\eta } \\
			\approx v^0\int_{0}^\infty dr(\sqrt{1+r^2})^{-\eta }\int_{0}^\pi d\phi \int_{0}^{2\pi}d\theta\  r^2\sin\phi\hspace{1mm}u(v^0+\sqrt{1+r^2}-v'^0)\\\times \delta\left(\frac{1}{2\bar{g}}(\bar{g}^2-2\sqrt{1+r^2}(v^0-v'^0)+2r|v-v'|\cos\phi)\right) \\
			\approx  v^0\int_{0}^\infty dr (\sqrt{1+r^2})^{-\eta }\int_{-1}^1 d(-\cos\phi) \int_{0}^{2\pi}d\theta\  r^2\hspace{1mm}u(v^0+\sqrt{1+r^2}-v'^0)\\\times \frac{\bar{g}}{r|v-v'|}\delta\left(\cos\phi+\frac{\bar{g}^2-2\sqrt{1+r^2}(v^0-v'^0)}{2r|v-v'|}\right) \\
			\lesssim  v^0\int_{0}^\infty dr (\sqrt{1+r^2})^{-\eta }r\int_{0}^{2\pi}d\theta \ u(v^0+\sqrt{1+r^2}-v'^0) \frac{\bar{g}}{|v-v'|} \\
			\lesssim  v^0\int_{0}^\infty dr (\sqrt{1+r^2})^{-\eta }r \lesssim  v^0,
			\end{multline*}
				where we have used $s\lesssim v^0v_*^0$, $|u(x)|\leq 1$, $\bar{g}\leq |v-v'|$, and $\eta > 2$.  Therefore, we conclude that
				$$K\lesssim \left(\int_\rth \frac{dv}{v^0}\int_\rth \frac{dv'}{v'^0} v^0|f(v)||\psi(v')|^{p'}\right)^{\frac{1}{p'}}\lesssim \|f\|^{\frac{1}{p'}}_{L^1}\|\psi\|_{L^{p'}}.$$ Then we use $||\psi||_{L^{p'}}\leq 1$ and obtain
				$$K\lesssim \left\| f\right\|_{L^1}^{\frac{1}{p'}}.$$ Since $\frac{1}{p}+\frac{1}{p'}=1$, this completes the proof by \eqref{middlestep}.
\end{proof}This theorem immediately implies the following corollary:
\begin{corollary}
Choose and fix any $q\in(1,+\infty)$. Then we have
\begin{equation}\label{L1bound}
||Q^+(f,h)||_{L^1}\lesssim  ||f||_{L^1_{\frac{1}{2}}}||h||_{L^1_{\frac{1}{2}}}\lesssim ||f||_{L^1_{\frac{1}{2}}}||h||_{L^{1}_{(2q-1)\eta+\frac{1}{2} }},
\end{equation}and \begin{equation}\label{Linftybound}
||Q^+(f,h)||_{L^{2q}}\lesssim ||f||_{L^1_{\frac{1}{2}}}||h||_{L^{2q}_{(2q-1)\eta+\frac{1}{2} }}.
\end{equation}\end{corollary}\begin{proof}
This is a direct consequence of Theorem \ref{Lp estimate Q+} with $p=1$ and $p=2q$. 
\end{proof}
We can now use the Riesz-Thorin interpolation theorem to obtain the following improved estimate on the gain operator $Q^+$:
\begin{corollary}\label{Q+bound}
Let $q>1$. For $\eta >2$, we have
\begin{equation}\label{rieszthorin}
||Q^+(f,h)||_{L^q}\lesssim ||f||_{L^1_{\frac{1}{2}}}||h||_{L^n_{(2q-1)\eta+\frac{1}{2} }},\end{equation}
where
\begin{equation}n=\left\lbrace\begin{aligned}
&\frac{5q}{3+2q}\qquad \text{if}\ q\in(1,6],\\ &\frac{q(q-3)}{2q-3}\qquad \text{if}\ q\in [6,+\infty).
\end{aligned}\right.
\end{equation}
\end{corollary}\begin{proof}
This follows by the Riesz-Thorin interpolation theorem between \eqref{L1bound} and \eqref{Gain of integrability} for the case $q\in (1,6]$ and between \eqref{Linftybound} and \eqref{Gain of integrability} for the case $q\in [6,+\infty)$.
\end{proof}
\section{Propagation of $L^p$ estimates}\label{Section proof of the main theorem}
This section is devoted to prove the propagation of $L^p$ integrability of the solutions to the spatially homogeneous relativistic Boltzmann equation. We may now prove our main theorem.
\begin{proof}[Proof of Theorem \ref{main}]
We start with multiplying $f^{p-1}$ to \eqref{homorelbol}. Then we obtain
\begin{equation}\label{mainineq1}
\frac{1}{p}\frac{d}{dt}||f||^p_{L^p}\leq \int_\rth  f^{p-1}Q^+(f,f)dv -\int_\rth f^{p-1}Q^-(f,f)dv.
\end{equation}
By Lemma \ref{lowerL}, we further have 
\begin{equation}
\label{mainineq2} -\int_\rth f^{p-1}Q^-(f,f)dv\leq -C_u\int_\rth (v^0)f^pdv =-C_u ||f||^p_{L^p_{1}}.
\end{equation}
On the other hand, by the H\"older inequality and Corollary \ref{Q+bound} with $m=p$, we have
$$
\int_\rth f^{p-1}Q^+(f,f)dv \leq \left\|f^{p-1}\right\|_{L^{p'}}\left\|Q^+(f,f)\right\|_{L^p}
\lesssim \left\|f\right\|^{p-1}_{L^{p}}||f||_{L^1_{\frac{1}{2}}}||f||_{L^n_{(2p-1)\eta+\frac{1}{2}}},
$$
where
\begin{equation}n=\left\lbrace\begin{aligned}
&\frac{5p}{3+2p}\qquad \text{if}\ p\in(1,6],\\ &\quad \frac{p(p-3)}{2p-3}\qquad \text{if}\ p\in [6,+\infty).
\end{aligned}\right.
\end{equation}Since the $L^n$ norm still contains some positive weights, we interpolate with the stronger norms. Interpolation between the Lebesgue spaces and the weights gives:
\begin{equation}
||f||_{L^n_{(2p-1)\eta+\frac{1}{2}}}\leq ||f||^\vartheta_{L^1_{\frac{1}{n}\left((2p-1)\eta+\frac{1}{2}\right)}}||f||^{1-\vartheta}_{L^p}, 
\end{equation}where we use  $\frac{1}{n}=\frac{\vartheta}{1}+\frac{1-\vartheta}{p}$ with 
\begin{equation}\vartheta =\left\lbrace\begin{aligned}
&\ \ \ \frac{2}{5}\qquad \text{if}\ p\in(1,6],\\ & \ \frac{p}{(p-1)(p-3)}\qquad \text{if}\ p\in [6,+\infty).
\end{aligned}\right.\end{equation}
Therefore, we have
\begin{equation}\label{mainineq3}\begin{split}
\int_\rth f^{p-1}Q^+(f,f)dv
&\lesssim ||f||_{L^1_{\frac{1}{2}}}||f||_{L^1_{\frac{1}{n}\left((2p-1)\eta+\frac{1}{2}\right)}} \|f\|^{p-\vartheta}_{L^{p}}\\&\lesssim||f_0||^2_{L^1_{\frac{1}{n}\left((2p-1)\eta+\frac{1}{2}\right)}} \|f\|^{p-\vartheta}_{L^{p}},\end{split}
\end{equation}where we used the propagation of $L^1$ moments with polynomial weights of arbitrary nonnegative powers \cite[Theorem 5.2, (2)]{MR3166961}. Here, the number of weights $m$ is defined as 
\begin{equation}\label{mdef}m\eqdef \frac{1}{n}\left((2p-1)\eta+\frac{1}{2}\right)=\left\lbrace\begin{aligned}
&\frac{(3+2p)(2p-1)}{5p}\eta + \frac{3+2p}{10p},\quad \text{if}\ p\in(1,6],\\ &\frac{(2p-3)(2p-1)}{p(p-3)}\eta+\frac{2p-3}{2p(p-3)},\quad \text{if}\ p\in [6,+\infty).
\end{aligned}\right.
\end{equation}Then, the theorem follows by plugging \eqref{mainineq2} and \eqref{mainineq3} into \eqref{mainineq1}.
\end{proof}

\noindent{\bf Acknowledgements} J. W. Jang was supported by
the Korean IBS project IBS-R003-D1 and was partially supported by the Junior Trimester Program ``Kinetic Theory" of the Hausdorff Research Institute for Mathematics. S.-B. Yun is supported by Samsung Science and Technology Foundation under Project Number SSTF-BA1801-02.    


\bibliographystyle{plain}


\bibliography{bibliography}{}

\begin{thebibliography}{10}

\bibitem{MR1402446}
H{\aa}kan Andr\'{e}asson.
\newblock Regularity of the gain term and strong {$L^1$} convergence to
  equilibrium for the relativistic {B}oltzmann equation.
\newblock {\em SIAM J. Math. Anal.}, 27(5):1386--1405, 1996.

\bibitem{MR2102321}
H{\aa}kan Andr\'{e}asson, Simone Calogero, and Reinhard Illner.
\newblock On blowup for gain-term-only classical and relativistic {B}oltzmann
  equations.
\newblock {\em Math. Methods Appl. Sci.}, 27(18):2231--2240, 2004.

\bibitem{MR3300786}
A.~Bellouquid, J.~Nieto, and L.~Urrutia.
\newblock Global existence and asymptotic stability near equilibrium for the
  relativistic {BGK} model.
\newblock {\em Nonlinear Anal.}, 114:87--104, 2015.

\bibitem{MR2988960}
Abdelghani Bellouquid, Juan Calvo, Juanjo Nieto, and Juan Soler.
\newblock On the relativistic {BGK}-{B}oltzmann model: asymptotics and
  hydrodynamics.
\newblock {\em J. Stat. Phys.}, 149(2):284--316, 2012.

\bibitem{Cal}
Simone Calogero.
\newblock The {N}ewtonian limit of the relativistic {B}oltzmann equation.
\newblock {\em J. Math. Phys.}, 45(11):4042--4052, 2004.

\bibitem{C}
Carlo Cercignani.
\newblock {\em The {B}oltzmann equation and its applications}, volume~67 of
  {\em Applied Mathematical Sciences}.
\newblock Springer-Verlag, New York, 1988.

\bibitem{C-I-P}
Carlo Cercignani, Reinhard Illner, and Mario Pulvirenti.
\newblock {\em The mathematical theory of dilute gases}, volume 106 of {\em
  Applied Mathematical Sciences}.
\newblock Springer-Verlag, New York, 1994.

\bibitem{C-K}
Carlo Cercignani and Gilberto~Medeiros Kremer.
\newblock {\em The relativistic {B}oltzmann equation: theory and applications},
  volume~22 of {\em Progress in Mathematical Physics}.
\newblock Birkh\"auser Verlag, Basel, 2002.

\bibitem{DeGroot}
S.~R. de~Groot, W.~A. van Leeuwen, and Ch.~G. van Weert.
\newblock {\em Relativistic {K}inetic {T}heory. Principles and applications.}
\newblock North-Holland Publishing Co., Amsterdam-New York, 1980.

\bibitem{D}
Marek Dudy\'{n}ski.
\newblock On the linearized relativistic {B}oltzmann equation. {II}.
  {E}xistence of hydrodynamics.
\newblock {\em J. Statist. Phys.}, 57(1-2):199--245, 1989.

\bibitem{D-E0}
Marek Dudy\'{n}ski and Maria~L. Ekiel-Je\.{z}ewska.
\newblock Causality of the linearized relativistic {B}oltzmann equation.
\newblock {\em Phys. Rev. Lett.}, 55(26):2831--2834, 1985.

\bibitem{D-E0er}
Marek Dudy\'{n}ski and Maria~L. Ekiel-Je\.{z}ewska.
\newblock Errata: ``{C}ausality of the linearized relativistic {B}oltzmann
  equation''.
\newblock {\em Investigaci\'{o}n Oper.}, 6(1):2228, 1985.

\bibitem{D-E3}
Marek Dudy\'{n}ski and Maria~L. Ekiel-Je\.{z}ewska.
\newblock On the linearized relativistic {B}oltzmann equation. {I}. {E}xistence
  of solutions.
\newblock {\em Comm. Math. Phys.}, 115(4):607--629, 1988.

\bibitem{D-E2}
Marek Dudy\'{n}ski and Maria~L. Ekiel-Je\.{z}ewska.
\newblock Global existence proof for relativistic {B}oltzmann equation.
\newblock {\em J. Statist. Phys.}, 66(3-4):991--1001, 1992.

\bibitem{E-M-V}
Miguel Escobedo, St\'{e}phane Mischler, and Manuel~A. Valle.
\newblock {\em Homogeneous {B}oltzmann equation in quantum relativistic kinetic
  theory}, volume~4 of {\em Electronic Journal of Differential Equations.
  Monograph}.
\newblock Southwest Texas State University, San Marcos, TX, 2003.

\bibitem{GS4}
R.~T. Glassey and W.~A. Strauss.
\newblock Asymptotic stability of the relativistic {M}axwellian via fourteen
  moments.
\newblock {\em Transport Theory Statist. Phys.}, 24(4-5):657--678, 1995.

\bibitem{Gla}
Robert~T. Glassey.
\newblock {\em The {C}auchy problem in kinetic theory}.
\newblock Society for Industrial and Applied Mathematics (SIAM), Philadelphia,
  PA, 1996.

\bibitem{GL-Vacuum}
Robert~T. Glassey.
\newblock Global solutions to the {C}auchy problem for the relativistic
  {B}oltzmann equation with near-vacuum data.
\newblock {\em Comm. Math. Phys.}, 264(3):705--724, 2006.

\bibitem{GS3}
Robert~T. Glassey and Walter~A. Strauss.
\newblock Asymptotic stability of the relativistic {M}axwellian.
\newblock {\em Publ. Res. Inst. Math. Sci.}, 29(2):301--347, 1993.

\bibitem{Guo-Strain2}
Yan Guo and Robert~M. Strain.
\newblock Momentum regularity and stability of the relativistic
  {V}lasov-{M}axwell-{B}oltzmann system.
\newblock {\em Comm. Math. Phys.}, 310(3):649--673, 2012.

\bibitem{H-R-Yun}
Byung-Hoon Hwang, R.~Tomasso, and Seok-Bae Yun.
\newblock In preparation.

\bibitem{1811.10023}
Byung-Hoon Hwang and Seok-Bae Yun.
\newblock Anderson-{W}itting model of the relativistic {B}oltzmann equation
  near equilibrium, 2018.

\bibitem{1801.08382}
Byung-Hoon Hwang and Seok-Bae Yun.
\newblock {S}tationary solutions to the boundary value problem for relativistic
  {B}{G}{K} model in a slab, 2018.

\bibitem{Jang2016}
Jin~Woo Jang.
\newblock {\em {G}lobal classical solutions to the relativistic {B}oltzmann
  equation without angular cut-off}.
\newblock PhD thesis, University of Pennsylvania, 2016.
\newblock (ProQuest Document ID 1802787346).

\bibitem{Jang-Strain}
Jin~Woo Jang and Robert~M. Strain.
\newblock On the determinant problem for the relativistic boltzmann equation,
  in preparation.

\bibitem{Jang-Strain-Yun}
Jin~Woo Jang, Robert~M. Strain, and Seok-Bae Yun.
\newblock Propagation of uniform upper bounds for the spatially homogeneous
  relativistic boltzmann equation, in preparation.

\bibitem{MR3880739}
Jin~Woo Jang and Seok-Bae Yun.
\newblock Gain of regularity for the relativistic collision operator.
\newblock {\em Appl. Math. Lett.}, 90:162--169, 2019.

\bibitem{MR2378164}
Zhenglu Jiang.
\newblock Global existence proof for relativistic {B}oltzmann equation with
  hard interactions.
\newblock {\em J. Stat. Phys.}, 130(3):535--544, 2008.

\bibitem{MR3169776}
Ho~Lee and Alan~D. Rendall.
\newblock The spatially homogeneous relativistic {B}oltzmann equation with a
  hard potential.
\newblock {\em Comm. Partial Differential Equations}, 38(12):2238--2262, 2013.

\bibitem{MR2081030}
Cl\'{e}ment Mouhot and C\'{e}dric Villani.
\newblock Regularity theory for the spatially homogeneous {B}oltzmann equation
  with cut-off.
\newblock {\em Arch. Ration. Mech. Anal.}, 173(2):169--212, 2004.

\bibitem{Pennisi_2018}
Sebastiano Pennisi and Tommaso Ruggeri.
\newblock A new {BGK} model for relativistic kinetic theory of monatomic and
  polyatomic gases.
\newblock {\em Journal of Physics: Conference Series}, 1035:012005, 2018.

\bibitem{MR2728733}
Robert~M. Strain.
\newblock Asymptotic stability of the relativistic {B}oltzmann equation for the
  soft potentials.
\newblock {\em Comm. Math. Phys.}, 300(2):529--597, 2010.

\bibitem{MR2679588}
Robert~M. Strain.
\newblock Global {N}ewtonian limit for the relativistic {B}oltzmann equation
  near vacuum.
\newblock {\em SIAM J. Math. Anal.}, 42(4):1568--1601, 2010.

\bibitem{MR2765751}
Robert~M. Strain.
\newblock Coordinates in the relativistic {B}oltzmann theory.
\newblock {\em Kinet. Relat. Models}, 4(1):345--359, 2011.

\bibitem{Guo-Strain3}
Robert~M. Strain and Yan Guo.
\newblock Stability of the relativistic {M}axwellian in a collisional plasma.
\newblock {\em Comm. Math. Phys.}, 251(2):263--320, 2004.

\bibitem{StrainTas}
Robert~M. Strain and Maja Taskovi{\'{c}}.
\newblock {E}ntropy dissipation estimates for the relativistic {L}andau
  equation, and applications.
\newblock {\em J. Funct. Anal.}, pages 1--50, 2019.

\bibitem{1903.05301}
Robert~M. Strain and Zhenfu Wang.
\newblock Uniqueness of bounded solutions for the homogeneous relativistic
  {L}andau equation with {C}oulomb interactions.
\newblock {\em Quart. Appl. Math.}, in press:1--39, 2019.

\bibitem{MR3166961}
Robert~M. Strain and Seok-Bae Yun.
\newblock Spatially homogeneous {B}oltzmann equation for relativistic
  particles.
\newblock {\em SIAM J. Math. Anal.}, 46(1):917--938, 2014.

\bibitem{Vil02}
C{\'{e}}dric Villani.
\newblock A review of mathematical topics in collisional kinetic theory.
\newblock In {\em Handbook of mathematical fluid dynamics, {V}ol. {I}}, pages
  71--305. North-Holland, Amsterdam, 2002.

\bibitem{MR1450762}
Bernt Wennberg.
\newblock Entropy dissipation and moment production for the {B}oltzmann
  equation.
\newblock {\em J. Statist. Phys.}, 86(5-6):1053--1066, 1997.

\end{thebibliography}





\end{document}